\documentclass[11pt,hyp,]{nyjm}
\usepackage{todonotes}
\usepackage{sseq,pgffor}
\usepackage{tikz-cd} 

\usepackage{amsmath,amsthm,amssymb,array,enumerate,parskip,graphicx,etoolbox,tabularx}
\usepackage[all]{xy}
\usepackage{mathtools}
\usepackage{xcolor} 
\usepackage[pagebackref]{hyperref}
\definecolor{dark-red}{rgb}{0.5,0.15,0.15}
\definecolor{dark-blue}{rgb}{0.15,0.15,0.6}
\definecolor{dark-green}{rgb}{0.15,0.6,0.15}
\hypersetup{
    colorlinks, linkcolor=dark-red,
    citecolor=dark-blue, urlcolor=dark-green
}

\usepackage[nameinlink,capitalise,noabbrev]{cleveref}

\numberwithin{equation}{section}

\newcount\tmp
\newcommand{\F}{\mathbb{F}}
\newcommand{\Z}{\mathbb{Z}}

\newcommand{\C}{\mathbb{C}}

\newcommand{\xr}{\xrightarrow}

\newcommand{\cal}{\mathcal}

\usepackage{comment}

\makeatletter
\newcommand*{\myfnsymbolsingle}[1]{%
  \ensuremath{%
    \ifcase#1
    \or 
      *%
    \or 
      \dagger
    \or 
      \ddagger
    \or 
      \mathsection
    \or 
      \mathparagraph
    \else 
      \@ctrerr  
    \fi
  }%
}   
\makeatother

\renewcommand*{\backref}[1]{}
\renewcommand*{\backrefalt}[4]{%
  \ifcase #1 %
No citations.
  \or
(cit. on p. #2).%
  \else
(cit on pp. #2).%
  \fi%
}

\usepackage{alphalph}
\newalphalph{\myfnsymbolmult}[mult]{\myfnsymbolsingle}{}
\newcommand{\M}{\mathbb{M}}

\DeclareMathOperator{\Spec}{Spec}

\DeclareMathOperator{\im}{im}

\DeclareMathOperator{\Hom}{Hom}
\DeclareMathOperator{\Ext}{Ext}
\DeclareMathOperator{\Tor}{Tor}

\newcommand{\floor}[1]{\lfloor #1 \rfloor}

\newtheorem{thm}{Theorem}[section]

\newtheorem{lem}[thm]{Lemma} 
\newtheorem{cor}[thm]{Corollary}
\newtheorem{exmp}[thm]{Example}
\newtheorem{prop}[thm]{Proposition} \theoremstyle{definition}
\newtheorem{rem}[thm]{Remark} 
\newtheorem{defn}[thm]{Definition}
\newtheorem{cond}{Condition}[section]

\Crefname{cor}{Corollary}{Corollaries}
\Crefname{conjecture}{Conjecture}{Conjectures}
\Crefname{rem}{Remark}{Remarks}
\Crefname{prop}{Proposition}{Propositions}  
\Crefname{question}{Question}{Questions}
\Crefname{figure}{Figure}{Figures}
\Crefname{cond}{Condition}{Conditions}

\title{Vanishing lines for modules over the motivic Steenrod algebra}
\author{Drew Heard}
\address{Fachbereich Mathematik der Universit{\"a}t Hamburg, Bundesstra{\ss}e 55, 20146 Hamburg, Germany}
\email{drew.heard@uni-hamburg.de}
\author{Achim Krause}
\address{Max-Planck-Institute for Mathematics, Vivatsgasse 7, 53111 Bonn, Germany}
\email{achim@mpim-bonn.mpg.de}
\date{\today}

\newcommand{\cA}{\mathcal{A}}

\DeclareMathOperator{\Sq}{Sq}
\begin{document}

\begin{abstract}We study criteria for freeness and for the existence of a vanishing line for modules over certain Hopf subalgebras of the motivic Steenrod algebra over $\Spec \C$ at the prime 2. These turn out to be determined by the vanishing of certain Margolis homology groups in the quotient Hopf algebra $\cA/\tau$. 
\end{abstract}
\maketitle
\section{Introduction}
Let $\cA^{\text{cl}}$ be the classical mod $p$ Steenrod algebra. Criteria for freeness and for the existence of a vanishing line for modules over Hopf subalgebras of $\cA^{\text{cl}}$ have been obtained by Adams--Margolis \cite{am_sub,am_modules}, Moore--Peterson \cite{mp_modules}, Anderson--Davis \cite{ad_vanishing}, and Miller--Wilkerson \cite{mw_vanishing}, with special cases also considered by other authors. These results are of great importance; for example, they are used critically in the proof of the nilpotence and periodicity theorems of Devinatz, Hopkins, and Smith \cite{dhs_nilpotence,hs_periodicity}. 

Now let $\cA$ be the motivic Steenrod algebra. We work over $\Spec \C$ (although the results also hold more generally for algebraically closed fields,  see \Cref{rem:otherfields}) and at $p=2$, since the odd-primary motivic Steenrod algebra is a base-changed version of the classical odd-primary Steenrod algebra. By work of Voevodsky \cite{voe_1}, as an algebra, the dual $\cA_{\ast,\ast}$ of $\cA$ is given by
\[
\cal{A_{\ast,\ast}} \cong \M_2[\xi_{i+1},\tau_i \mid i \ge 0]/(\tau_i^2 + \tau \xi_{i+1}),
\]
where $\M_2 \cong \F_2[\tau]$ is the motivic cohomology of a point (see \Cref{sec:background} for more details on $\cA$ and $\cA_{\ast,\ast})$.
\begin{rem}
  We remind the reader that motivic objects are bigraded; they have internal degrees $q$ (the topological degree) and $w$ (the weight). However, the weight will not play a major role in this paper, and we will usually suppress it from the notation, considering motivic objects as singly graded. 
\end{rem}

 As shown in \cite{gir_picard} and \cite{gi_vanishing} the most naive adaptation of the classical results on freeness and vanishing lines cannot hold for $\cA$. Indeed, note the following, where we refer the reader to \cite[Ch.~19 and Ch.~20]{mar_spectra} for the definition of Margolis homology groups:
\begin{enumerate}
  \item Let $\cA(1)^{\text{cl}}$ be the Hopf subalgebra of the 2-primary classical Steenrod algebra $\cA^{\text{cl}}$ generated by $\Sq^1$ and $\Sq^2$, and $\cA(1)$ its motivic analog. Then an $\cA(1)^{\text{cl}}$-module $M$ is projective (equivalently, free, since $\cA(1)^{\text{cl}}$ is a connective Hopf algebra over a field, see \cite[Prop.~11.2.2]{mar_spectra}) if and only if the Margolis homology groups $H(M;Q_0)$ and $H(M;Q_1)$ vanish. In \cite[Ex.~4.6]{gir_picard} the authors give an example of an $\cA(1)$-module with vanishing $H(M;Q_0)$ and $H(M;Q_1)$ that is not projective. Furthermore, in Prop.~4.7 of \emph{loc.~cit.~}~they show that a third Margolis homology group must vanish, as well as requiring $M$ to be a free $\M_2$-module, in order for $M$ to be a projective $\cA(1)$-module (note that the latter condition is automatic classically, since we work over $\F_2$).
  \item  Let $\cA(0)^{\text{cl}}$ be the Hopf subalgebra of the 2-primary classical Steenrod algebra $\cA^{\text{cl}}$ generated by $\Sq^1$, and $\cA(0)$ its motivic analog. 
         Adams \cite[Prop.~2.5]{adams_periodicity} proves that if $M$ is a bounded-below $\cA^{\text{cl}}$-module which is free over $\cA(0)^{\text{cl}}$, then $\Ext^{p,q}_{\cA^{\text{cl}}}(M,\F_2)$ has a vanishing line of slope 1/2 
         (that is, when drawn in the $(q-p,p)$ plane, the $\Ext$ groups vanish above a line of slope 1/2). 
         Motivically, Guillou and Isaksen \cite[Prop.~7.2]{gi_vanishing} prove that if $N$ is a bounded-below $\cA$-module which is free over $\cA(0)$, then above a line of slope $1/2$, all the elements of $\Ext^{p,q}_{\cA}(N,\F_2)$ are $h_1$-local (here again, the slope refers to only the $(q-p,p)$ plane. They also observed that there is a vanishing line, but now of slope $1$; that is, $\Ext^{p,q}_{\cA}(N,\F_2) = 0$ whenever $p > ( q-p) + c$ for some constant $c$. 
\end{enumerate}

The goal of this paper is to explain results such as these. The key observation, in a sense already noted in \cite{gir_picard}, is that it is enough to work with the quotient Hopf algebra $\cA/\tau$. We shall see that $\cA/\tau$ is a (bigraded) 2-primary version of the classical odd-primary Steenrod algebra. This already demystifies the above results partially; for odd primes, freeness of an $\cA(1)^{\text{cl}}$-module does require the vanishing of three Margolis homology groups. We shall also see how to explain the vanishing line of slope 1, as well as the $h_1$-local part, see \Cref{ex:gi_vanishing}. In future work with Tobias Barthel we will use our main theorem to investigate periodicity and vanishing line results in motivic homotopy theory akin to this example. 

Classically Adams and Margolis \cite{am_sub} have shown that all Hopf subalgebras of the odd-primary $\cA^{\text{cl}}$ are the obvious ones; namely the duals of the quotients 
\[
\cA^{\text{cl}}_{\ast}/(\xi_1^{p^{h(1)}},\xi_2^{p^{h(2)}},\ldots,\tau_0^{2^{k(0)}},\tau_1^{2^{k(1)}},\ldots)
\]
where $0 \le h(i) \le \infty$ and $0 \le k(i) \le 1$, and the sequences $h = (h(1),h(2),\ldots)$ and $k=(k(1),k(2),\ldots)$ must satisfy certain conditions, see Condition 2.1 and 2.2 of \cite{am_sub}. The pair $(h,k)$ is called the profile function of the Hopf subalgebra. The first step in our work is to construct certain Hopf subalgebras of $\cA$ by a similar method. There is an added complexity arising from the fact that $\M_2$ is not a field. Using the notion of a free profile function $(h,k)$, originally introduced in the $C_2$-equivariant Steenrod algebra by Ricka \cite{ricka_steenrod}, we construct certain quotient Hopf algebras $B_{\ast,\ast}(h,k)$ that are free as $\M_2$-modules. In fact every quotient Hopf algebra of $\cA_{\ast,\ast}$ that is free as an $\M_2$-module is of this form. 
\newtheorem{thmx}{Theorem}
\renewcommand{\thethmx}{\Alph{thmx}} 
\begin{thmx}
  A quotient Hopf algebra of $\cA_{\ast,\ast}$ is free as an $\M_2$-module if and only if it is of the form $B_{\ast,\ast}(h,k)$ for a profile function satisfying certain numerical conditions.
\end{thmx}
For a precise version of the theorem, see \Cref{thm:freeclassification}. 

In order to state our main result, let $P_t^s$ for $s,t>0$ be dual to $\xi_t^{2^{s-1}}$ and $Q_t$ be dual to $\tau_t$. We shall see that these also make sense in the quotient $\cA/\tau$, and there they satisfy $(P_t^s)^2 = 0$ for $s\leq t$ and $(Q_t)^2 = 0$, so we can define Margolis homology groups $H(M/\tau;Q_t)$ and $H(M/\tau;P_t^s)$ for an $\cA$-module $M$. For the following, compare the main theorem of \cite{mw_vanishing}. Note that we write $|P_t^s|$ for the topological degree. The following is given as \Cref{thm:mainthm}.

\begin{thmx}
    Let $B=B(h,k)$ be a Hopf subalgebra of $\cA$ which is free as an $\M_2$-module. Let $M$ be a $B$-module that is free and of finite type over $\M_2$. 
  \begin{enumerate}[(i)]
    \item If $H(M/\tau;P_t^s) =0$ for all $P_t^s \in B$ with $s\leq t$ and $H(M/\tau;Q_t) = 0$ for all $Q_t \in B$, then $M$ is $B$-free. 
    \item Let $d>0$ be an integer. If $H(M/\tau;P_t^s) =0$ for all $P_t^s \in B$ such that $s\leq t$ and $|P_t^s| < d$ and if $H(M/\tau;Q_t) = 0$ for all $Q_t$ such that $|Q_t| < d$, then $\Ext_B^{p,q}(M,\M_2) = 0$ for all $(p,q)$ such that $q < dp-c$ with $c$ depending only on $d$ and the connectivity of $M$. 
  \end{enumerate}
\end{thmx}
We note that if $M$ is a $\cA$-module, which is free and of finite type over $\M_2$, such that $x\in \cA$ has the property that $x^2$ acts trivially on $M$, then $H(M;x) = 0$ if and only if $H(M/\tau;x) = 0$  see \Cref{lem:marcomparasion}. 

This paper is organized as follows: we first give an introduction to the motivic Steenrod algebra. We then construct certain Hopf subalgebras of $\cA$. The proof of our main theorem is carried out in \Cref{sec:proof}, and relies on a reduction to $\cA/\tau$, which is studied in more detail in \Cref{sec:amodtau}. 
\section*{Acknowledgments}
The first author is grateful to the Max Planck Institute for Mathematics and the Universit{\"a}t Hamburg for hospitality, and to the Max Planck Institute and DFG Schwerpunktprogramm SPP 1786 for financial support. The second author is grateful to the Max Planck Institute for Mathematics and the International Max Planck Research School for hospitality and financial support. The authors thank Dan Isaksen and the anonymous reviewer for helpful comments on earlier versions of this document. 
\section{Background}\label{sec:background}
Let $\mathcal{SH}(\C)$ be the homotopy category of motivic spectra over $\Spec(\C)$, and let $p=2$. We denote by $\M_2 \cong \F_2[\tau]$ the motivic cohomology of a point, where $|\tau|$ has bidegree $(0,1)$. Let $\cal{A}$ denote the motivic Steenrod algebra, and $\cal{A}_{\ast,\ast} = \Hom_{\M_2}(\cA,\M_2)$ its dual. The properties of $\cA$ and $\cA_{\ast,\ast}$ have been determined by Voevodsky in \cite{voe_1,voe_2},   We give a summary of what we need here - see in particular \cite[Prop.~10.2, Thm.~12.6, and Lem.~12.11]{voe_1}, with some typos in the Adem relations corrected in \cite[Thm.~4.5.1]{1207.3121}. In order to state these results, we note that a sequence $I = (i_1,\ldots,i_k)$ of integers is called admissible if either $I = \emptyset$ or $k \ge 1,i_k \ge 1$ and $i_{j-1} \ge 2 i_j$ for all $k \ge j \ge 2$. 
\begin{thm}[Voevodsky]\label{thm:voev}
  \begin{enumerate}[(i)]
\item The motivic Steenrod algebra $\cA$ is the associative algebra over $\M_2$ generated by $\Sq^{2i} \in (2i,i)$ and $\Sq^{2i-1} \in (2i-1,i-1)$ for $i \ge 1$ subject to the following relations, where $0 < a < 2b$:
\begin{equation}\label{eq:motadem}
\Sq^{a}\Sq^{b} = \sum_{j=0}^{\floor{a/2}}\tau^{e_i}\binom{b-1-j}{a-2j}\Sq^{a+b-j}\Sq^{j}
\end{equation}
where
\[
e_i = \begin{cases}
  1 & \text{ if } a,b \text{ even, and } j \text{ odd} \\
  0 & \text{ else.}
\end{cases}
\]
\item   The motivic Steenrod algebra $\cA$ is a free $\M_2$-module on the admissible monomials $\Sq^I = \Sq^{i_1}\ldots\Sq^{i_k}$.
\item The dual motivic Steenrod algebra $\cA_{\ast,\ast}= \Hom_{\M_2}(\cA,\M_2)$ is given as an algebra by 
\[
\cal{A_{\ast,\ast}} \cong \M_2[\tau_i,\xi_{i+1} \mid i \ge 0]/(\tau_i^2 + \tau \xi_{i+1}),
\]
where the bigrading is given by $|\xi_i| = (2^{i+1}-2,2^i-1)$, and $|\tau_i| = (2^{i+1}-1,2^i-1)$.
\item The coproduct in $\cA_{\ast,\ast}$ is given by 
\[
\begin{split}
\psi_*(\xi_k) &= \sum_{i=0}^k \xi_{k-i}^{2^i} \otimes \xi_i \\
\psi_*(\tau_k) &= \sum_{i=0}^k \xi_{k-i}^{2^i} \otimes \tau_i + \tau_k \otimes 1. 
\end{split}
\]
\end{enumerate}
\end{thm}

\begin{rem}\label{rem:tau}
\begin{enumerate}[(i)]
 \item $\Sq^1$ can be identified with the Bockstein $\beta$ associated to the short exact sequence
\[
0 \to \Z/2 \to \Z/4 \to \Z/2 \to 0. 
\]
Note that the Adem relations show that $\Sq^{2i+1} = \beta\Sq^{2i}$ for all $i \ge 0$, where $\Sq^0=1$.
\item Note that $\cA_{\ast,\ast} = \Hom_{\M_2}(\cA,\M_2)$ is given the homological grading, so $\tau$ acts by bidegree $(0,-1)$ on it. 
\end{enumerate}
\end{rem}

From the description of the dual, we immediately see:
\begin{lem}\label{lem:haequivalence}
After inverting $\tau$, there is an isomorphism of Hopf algebras $\tau^{-1}\cA_{*,*} \cong \tau^{-1}\M_2\otimes_{\F_2}\cA^{\mathrm{cl}}_{*}$, where $\cA^{\mathrm{cl}}_{*}$ is the classical dual Steenrod algebra. The isomorphism is given by 
\begin{align*}
\tau_k &\mapsto \tau^{1-2^k}\xi_{k+1} \\
\xi_k &\mapsto \tau^{1-2^k}\xi_k^2
\end{align*}
Dually, there is an isomorphism $\tau^{-1}\cA \cong \tau^{-1}\M_2\otimes_{\F_2}\cA^{\mathrm{cl}}$.
\end{lem}

\begin{defn}
We will refer to the embedding $\cA_{*,*}\to \tau^{-1}\M_2\otimes_{\F_2}\cA^{\mathrm{cl}}_{*}$ as the embedding into the classical dual Steenrod algebra, since its codomain can be thought of as the classical dual Steenrod algebra over the graded field $\tau^{-1}\M_2$. Similarly, we refer to the embedding $\cA\to \tau^{-1}\M_2\otimes_{\F_2}\cA^{\mathrm{cl}}$ as the embedding into the classical Steenrod algebra.
\end{defn}

The following is straightforward and can be proved in the same way as the classical odd-primary case.
\begin{lem}\label{lem:dualfree}
  Let $E=(e_0,e_1,\ldots)$ be a sequence of ones and zeroes, almost all zero, and $R=(r_1,r_2,\ldots)$ a sequence of nonnegative integers, almost all zero. As an $\M_2$-module, $\cA_{\ast,\ast}$ is free with basis 
  \[
\cal{B}_m \coloneqq \{ \Pi_{i \ge 0} \tau_i^{e_i} \Pi_{j \ge 1} \xi_j^{r_j}\}.
  \]
\end{lem}
\begin{defn}
  Let $\rho(E,R)$ be the dual monomial basis of $\cA$. Define
\[
P^R = \rho(0,R)
\]
and
\[
Q(E) = \rho(E,0).
\]
Additionally let $Q_i$ be dual to $\tau_i$, i.e., 
\[
Q_i = Q(0,\ldots,1,\ldots),
\]
where the 1 is in the $i$-th spot. Finally, for $s>1$, define $P_t^s = P(0,\ldots,2^{s-1},\ldots)$, where the $2^{s-1}$ is in the $t$-th position, to be the class dual to $\xi_t^{2^{s-1}}$ (note that the indexing on $Q$ starts at position 0, and that on $P$ starts at position 1).
\end{defn}

 The indexing is chosen such that the motivic $P_s^t$ maps to the classical one under the natural isomorphism $\F_2 \otimes_{\M_2} \cA_{\ast,\ast} \simeq \cA_{\ast}$. More precisely, we have:
 \begin{lem}
Under the embedding $\cA \hookrightarrow \tau^{-1}\M_2\otimes_{\F_2}\cA^{\mathrm{cl}}$ into the classical Steenrod algebra, $\rho(E,R)$ goes to a $\tau$-multiple of the classical Milnor basis element $P^{E+2R}$. 
 \end{lem}
\begin{proof}
Dually, this corresponds to the fact that monomials in the motivic dual Steenrod algebra go to $\tau$-multiples of the corresponding monomials in the classical dual Steenrod algebra, which is immediate from the definition of the embedding.
\end{proof}
\begin{rem}
This enables us to recover multiplicative relations between the $\rho(E,R)$ from the corresponding relations between the $P^{E+2R}$, since the powers of $\tau$ that occur are determined by weight. For example, the $\C$-motivic Adem relations can be deduced from this. 

Furthermore, one can also use this to deduce that the classical recursion formula for the $Q_i$ given by
\[
Q_{n+1} = [Q_n,\Sq^{2^{n+1}}],
\]
also holds in the motivic Steenrod algebra. This also appears in \cite[Eq.~(2)]{Kylling} over more general base fields (in general, the result requires a correction term when $\sqrt{-1} \not \in k$).

\end{rem}

\begin{lem}
  We have $(Q_i)^2 = 0$ and $(P^s_t)^2 = 0$ whenever $s<t$. 
\end{lem}
\begin{proof}
This also follows immediately from the corresponding relations in the classical Steenrod algebra.
\end{proof}
\begin{rem}
For $s=t$, $(P_t^s)^2$ is nonzero. For example, at $p=2$, $P_1^1=\Sq^2$, and $\Sq^2\Sq^2 = \tau \Sq^3 \Sq^1$ (note that non-motivically, we have $(P_1^1)^2 = \Sq^3 \Sq^1$). Note though that $(P^1_1)^2 \equiv 0 \mod (\tau)$; we shall see in \Cref{sec:amodtau} that this type of result holds for all $P_t^s$ with $s=t$.
\end{rem}

\section{Some quotient Hopf algebras of $\cA_{\ast,\ast}$}
In this section, we introduce certain quotient Hopf algebras of $\cA_{\ast,\ast}$, and identify conditions that ensure that they are free as $\M_2$-modules. The results are similar to the classification of quotient Hopf algebras of the classical odd-primary Steenrod algebra, although of course classically we do not need to worry about freeness. 

To begin, let $h$ be a function from the set $\{1,2,3,\ldots\}$ to the set $\{0,1,2,\ldots,\infty\}$, and let $k$ be a function from the set $\{0,1,2,\ldots\}$ to the set $\{0,1,2,\ldots,\infty\}$. We call the pair $(h,k)$ a \emph{profile function}. 
\begin{defn}
  The quotient algebra $B_{\ast,\ast}(h,k)$ is the quotient of $\cA_{\ast,\ast}$ by the relations
  \[
\begin{split}
\xi_i^{2^{h(i)}} &= 0 \quad (i=1,2,3,\ldots) \\
\tau_j^{2^{k(j)}} &= 0 \quad (j = 0,1,2,\ldots),
\end{split}
  \]
  with the convention that if $h(i)$ or $k(j) = \infty$, then we impose no relation. We let $I(h,k)$ denote the ideal of $\cA_{\ast,\ast}$ generated by these relations. 
\end{defn}

We point out two interesting properties in contrast to the classical Steenrod algebra: Firstly, $B_{\ast,\ast,}(h,k)$ need not be free as an $\M_2$ module, and, secondly, the relation $\tau_i^2 = \tau \xi_{i+1}$ in $\cA_{\ast,\ast}$ implies that two sequences $(h,i)$ and $(h',i')$ can give rise to isomorphic quotient algebras of $\cA_{\ast,\ast}$. 

We deal with the latter issue first, by defining a partial order on profile functions by $(h,k) \le (h',k')$ if we have $h(n+1) \le h'(n+1)$ and $k(n) \le k'(n)$ for all $n \ge 0$. This leads to the following definition, originally due to Ricka \cite[Def.~5.9]{ricka_steenrod}.
\begin{defn}
  A profile function is minimal if it is minimal among profile functions $(h',k')$ such that $I(h',k') = I(h,k)$. 
\end{defn}
 The following is then proved in the same way as \cite[Lem.~5.10]{ricka_steenrod}.
\begin{lem}
  A profile function is minimal if and only if for all $i,n \ge 0$, $\tau_i^{2^n} \in I(h,k)$ is equivalent to $n \ge k(i)$. 
   \end{lem}
We now introduce conditions that ensure that $B_{\ast,\ast}(h,k)$ is a quotient Hopf algebra of $\cA$. 
\begin{cond}\label{cond1}
  For all $i,j \ge 1$ we have $h(i) \le j + h(i+j)$ or $h(j) \le h(i+j)$.
  \end{cond}
  \begin{cond}\label{cond2}
  For all $i \ge 1,j \ge 0$ we have $h(i) \le j + k(i+j)$ or $k(j) \le k(i+j)$.   
  \end{cond}
For the following compare \cite[Prop.~5.13]{ricka_steenrod} or the odd-primary case of \cite{am_sub}. 
\begin{prop}\label{prop:quotientminimal}
  Let $(h,k)$ be a minimal profile functions satisfying \Cref{cond1,cond2}. Then the quotient algebra $B_{\ast,\ast} = B_{\ast,\ast}(h,k)$ is a quotient Hopf algebra of $\cal{A}_{\ast,\ast}$.  
\end{prop}
\begin{proof}
 We simply need to check that the coproduct passes through to the quotient. To that end, note that we have
  \[
\psi(\xi_n^{2^{h(n)}}) = \sum_{i+j=n} \xi_i^{2^{j+h(n)}} \otimes \xi_j^{2^{h(n)}} 
  \]
  and
  \[
\psi(\tau_n^{2^{k(n)}}) = \tau_n^{2^{k(n)}} \otimes 1 + \sum_{i+j=n} \xi_i^{2^{j+k(n)}} \otimes \tau_j^{2^{k(n)}}.
  \]
  Assuming \Cref{cond1}, then either $\xi_i^{2^{j+h(n)}} $ or $\xi_j^{2^{h(n)}}$ are in $I(h,k)$. By definition $\tau_n^{2^{k(n)}}$ is in $I(h,k)$, whilst if we assume \Cref{cond2} then either $ \xi_i^{2^{j+k(n)}}$ or $\tau_j^{2^{k(n)}}$ are in $I(h,k)$. Note that the latter condition is equivalent to $k(j) \le k(n)$ by minimality. It follows that $\psi$ passes to the quotient as required. 
\end{proof}

Recall that the classical Steenrod algebra is a Hopf algebra of finite type; that is, it is finitely generated as an $\F_2$-module in each degree. Since we are working over a field it is in addition automatically free as an $\F_2$-module in each degree. For such Hopf algebras, the dual inherits the structure of a Hopf algebra. One then easily checks that there is a one-to-one correspondence between quotient Hopf algebras of the dual Steenrod algebra, and Hopf subalgebras of the Steenrod algebra.  Indeed, as noted in the introduction, the classification of Hopf subalgebras of $\cA^{\text{cl}}$ is proved via the classification of all quotient Hopf algebras of $\cA^{\text{cl}}_{*}$. 

The situation in the motivic Steenrod algebra is more complicated since we work over the commutative ring $\M_2$. Since $\M_2$ is a graded principal ideal domain and $\cA_{\ast,\ast}$ is free as an $\M_2$-module, any Hopf subalgebra $B$ of $\cA$ is automatically free as an $\M_2$-module, and hence the dual quotient Hopf algebra of $\cA_{\ast,\ast}$ is a free $\M_2$-module. However, if we start with a quotient Hopf algebra $C_{\ast,\ast}$ of $\cA_{\ast,\ast}$, then, in general, the best we can conclude is that the dual has an $\M_2$-free cokernel. For this reason, we restrict ourselves to those quotient Hopf algebras which are free as $\M_2$-modules. Once again, our work is inspired by that of Ricka, although his result (in the context of the $C_2$-equivariant Steenrod algebra), is more complicated.
\begin{defn}
  A profile function $(h,k)$ is free if for all $i \ge 0$ we have 
  \[
  h(i+1) \le \begin{cases}
    k(i)-1, & k(i) \ne 0 \\
    0, & k(i) = 0. 
  \end{cases}
  \]
\end{defn}
\begin{rem}\label{rem:free}
\begin{enumerate}
\item One can easily check that if $(h,k)$ is a free profile function satisfying \Cref{cond1}, then \Cref{cond2} is automatically satisfied.   
\item   For a simple example of a profile function that is not free, but still satisfies \Cref{cond1,cond2}, one can take $h = k =(1,1,1,\ldots)$. 
  
\end{enumerate}

\end{rem}

Although we cannot give a complete classification of quotient Hopf algebras of $\cA_{\ast,\ast}$, the free profile functions do give a complete classification of those that are free as $\M_2$-modules. 

\begin{thm}\label{thm:freeclassification}
  A quotient Hopf algebra of $\cA_{\ast,\ast}$ is free as an $\M_2$-module if and only if it is of the form $B_{\ast,\ast}(h,k)$ for a free profile function satisfying \Cref{cond1,cond2}.
  \end{thm}
\begin{proof}

Let $\cA_{\ast,\ast}\to B_{\ast,\ast}$ be a quotient map, and $I \subset \cA_{\ast,\ast}$ its kernel. $B$ is free as an $\M_2$-module if and only if $I$ agrees with the kernel of $\cA_{\ast,\ast} \to \tau^{-1} B$, or equivalently, if $I$ agrees with the intersection of $\cA_{\ast,\ast}$ and the kernel of $\tau^{-1}\cA_{\ast,\ast} \to \tau^{-1} B$. 

Since $\tau^{-1}\M_2$ is periodic, we can recover this quotient map from its weight $0$ part. Under the isomorphism $\tau^{-1}\cA_{\ast,\ast} \cong \tau^{-1}\M_2\otimes_{\F_2}\cA^{\mathrm{cl}}_{*}$ of \Cref{lem:haequivalence}, this weight $0$ part gives a Hopf algebra quotient map out of the classical dual Steenrod algebra.

Adams and Margolis \cite{am_sub} showed that any Hopf algebra quotient of the classical dual Steenrod algebra is of the form $\cA^{\mathrm{cl}}_{\ast}/(\xi_i^{2^{h_{\text{cl}}(i)}})$. Here $i$ ranges over integers $\geq 1$, $h_{\text{cl}}(i)$ takes values in nonnegative integers or $\infty$, and for each $i,j$ we have $h_{\text{cl}}(i)\leq h_{\text{cl}}(i+j)+j$ or $h_{\text{cl}}(j)\leq h_{\text{cl}}(i+j)$.

From the description of the isomorphism $\tau^{-1}\cA_{\ast,\ast} \cong \tau^{-1}\M_2\otimes_{\F_2}\cA^{\mathrm{cl}}_{*}$ we see that the kernel of $\tau^{-1}\cA_{\ast,\ast} \to \tau^{-1} B$ is generated as an ideal in $\tau^{-1}\cA_{\ast,\ast}$ by $\tau_i^{2^{h_{\text{cl}}(i+1)}}$ and $\xi_i^{2^{h_{\text{cl}}(i)-1}}$.

This yields that the intersection of this kernel with $\cA_{\ast,\ast}$ is generated as an ideal in $\cA_{\ast,\ast}$ by $\tau_i^{2^{h_{\text{cl}}(i+1)}}$ and $\xi_i^{2^{h_{\text{cl}}(i)-1}}$, and so is of the form $I(h,k)$ for $k(i)=h_{\text{cl}}(i+1)$ and $h(i) = h_{\text{cl}}(i)-1$. \Cref{cond1} and \Cref{cond2} for $(h,k)$ of this form are equivalent to the classical condition for $h_{\text{cl}}$, and $(h,k)$ is obviously a free profile function. Vice-versa, any free profile function is determined by an $h_{\text{cl}}$ in this way.
\end{proof}

\begin{prop}
  Let $(h,k)$ be a free profile function satisfying \Cref{cond1,cond2}. Let $(e_0,e_1,\ldots)$ range over sequences of ones and zeroes, almost all zero, with $e_i < 2^{k(i)}$, and let $(r_1,r_2,\ldots)$ range over sequences of nonnegative integers, almost all zero, with $r_i < 2^{h(i)}$. Then, as an $\M_2$-module, $B_{\ast,\ast}(h,k)$ is free with basis
  \[
\cal{B}_{h,k} \coloneqq \{ \Pi_{i \ge 0} \tau_i^{e_i} \Pi_{j \ge 1} \xi_j^{r_j}\}. 
  \]
\end{prop}
\begin{proof}
We can again consider $B_{\ast,\ast}(h,k)$ as the image of $\cA_{\ast,\ast}$ in $\tau^{-1}\M_2\otimes_{\F_2} B^{\text{cl}}_{\ast}(h')$ with $h'(i)= k(i-1)$. Then the statement follows from the corresponding classical statement, which is that $B^{\text{cl}}_{\ast}(h')$ has basis the monomials 
\[
 \Pi_{i \ge 1} \xi_i^{r'_i}
\]
for $r'_i < h'(i)$.
\end{proof}
We now introduce a useful family of Hopf subalgebras of $\cA$. 
\begin{exmp}
   \normalfont Let $\cA(n)_{\ast,\ast}$ be the quotient Hopf algebra of $\cA_{\ast,\ast}$  defined by the minimal profile function $h = (n,n-1,n-2,\ldots,1,0,\ldots)$ and $k = (n+1,n,n-1,\ldots,1,0,\ldots)$. Since $h(i+1) = k(i)-1$ for $0 \le i \le n$ and both are 0 otherwise, the profile function is free, and $\cA(n)_{\ast,\ast}$ is a finitely generated free $\M_2$-module. Of course, one easily sees that 
   \[
\cA(n)_{\ast,\ast} \cong \M_2[\xi_1,\xi_2,\ldots,\xi_n,\tau_0,\tau_1,\ldots,\tau_n]/(\xi_i^{2^{n+1-i}},\tau_i^2 = \tau \xi_{i+1},\tau_n^2=0). 
  \]
It is dual to the subalgebra of $\cA$ generated by $\Sq^{2^i}$ for $i \le n$.
    \end{exmp}
\section{The Hopf algebra $\cA/\tau$}\label{sec:amodtau}
Let $\cA/\tau$ denote the quotient of the motivic Hopf algebra by $\tau$; note that this is a Hopf algebra over the field $\F_2 \cong \M_2/\tau$. In the introduction we claimed that this is a 2-primary version of the classical odd-primary Steenrod algebra. Indeed, one can check that in $\cA/\tau$ the Adem relations are the same as those in the classical odd-primary Steenrod algebra, under the identification $P^a = \Sq^{2a}$.  However, we do not need this result, rather we work with the dual $(\cA/\tau)_{\ast,\ast} = \Hom_{\F_2}(\cA/\tau,\M_2/\tau)$. We start with a general result about $\cA_{*,*}/\tau$-modules. 

\begin{lem}
If $M$ is an $\cA$-module, which is free as an $\M_2$-module, then there is an equivalence of $\cA_{\ast,\ast}/\tau$-modules between $(M/\tau)_{\ast,\ast} = \Hom_{\F_2}(M/\tau,\F_2)$ and $M_{\ast,\ast}/\tau = \Hom_{\M_2}(M,\M_2)/\tau$.
\end{lem}
\begin{proof}
    First note that 
  \[
  \begin{split}
(M/\tau)_{\ast,\ast} = \Hom_{\F_2}(M/\tau,\M_2/\tau) &\cong \Hom_{\M_2}(M/\tau,\M_2/\tau) \\
& \cong \Hom_{\M_2}(M,\M_2/\tau). 
\end{split}
  \]
  There is a short exact sequence
  \[
0 \to \M_2 \xr{\tau} \M_2 \to \M_2/\tau \to 0,
  \]
  and applying the exact functor $\Hom_{\M_2}(M,-)$ (since $M$ is $\M_2$-free) we get a short exact sequence
  \[
0 \to M_{\ast,\ast} \xr{\tau} M_{\ast,\ast} \to \Hom_{\M_2}(M,\M_2/\tau) \to 0.
  \]
   It follows that, as $\F_2$-modules, $(\cA/\tau)_{\ast,\ast} \cong \cA_{\ast,\ast}/\tau$, and it is easy to see that this is an isomorphism of $\cA_{*,*}/\tau$-modules. 
\end{proof}

Note that since $\tau$ does not appear in the formula for the comultiplication of $\cA_{\ast,\ast}$, the quotient $\cA_{\ast,\ast}/\tau$ has the structure of a Hopf algebra with comultiplication given by
\[
\begin{split}
\psi_*(\xi_k) &= \sum_{i=0}^k \xi_{k-i}^{2^i} \otimes \xi_i \\
\psi_*(\tau_k) &= \sum_{i=0}^k \xi_{k-i}^{2^i} \otimes \tau_i + \tau_k \otimes 1. 
\end{split}
\]
We then have the following. 
\begin{cor}\label{lem:dualident}
  As a Hopf algebra, the dual $(\cA/\tau)_{\ast,\ast}$ is isomorphic to 
   \[
\cA_{\ast,\ast}/\tau \cong \F_2[\xi_1,\ldots] \otimes \Lambda_{\F_2}(\tau_0,\ldots). 
\]
\end{cor}
\begin{proof}
We apply the previous lemma with $M = \cA$, which satisfies the conditions of the lemma by \Cref{thm:voev}(ii); we then just need to check that the isomorphism respects the Hopf algebra structures, but this is easy to do. 
\end{proof}

Recall that $P_s^t$ is dual to $\xi_t^{2^{s-1}}$ - this holds also in $\cA/\tau$. The degree-doubling isomorphism between $\cA_{\ast}^{\mathrm{cl}}$ and the quotient of $\cA_{\ast,\ast}$ by the $\tau_i$ allows one to prove the following, cf.~\cite[Prop.~2.3]{mp_modules}. 
\begin{cor}
  In $\cA/\tau$ we have $(P_s^t)^2 = 0$ for $s\leq t$. 
\end{cor}

Let $(h,l)$ be a profile function for $\cA_{\ast,\ast}/\tau$; note that since the classes $\tau_i$ are now exterior, $l$ is now only required to be a function from the set $\{ 0,1,2,\ldots\}$ to the set $\{0,1\}$.  Let $D_{\ast,\ast}(h,l)$ denote the quotient of $\cA_{\ast,\ast}/\tau$ by the ideal generated by the relations determined by $h$ and $l$, and write $D(h,l)$ for the corresponding Hopf subalgebra of $\cA/\tau$. Similar to \Cref{prop:quotientminimal}, we can define quotient Hopf algebras of $\cA/\tau$ by imposing conditions on the functions $h$ and $l$. 
\begin{cond}\label{cond1tau}
  For all $i,j \ge 1$ we have $h(i) \le j + h(i+j)$ or $h(j) \le h(i+j)$.
  \end{cond}
  \begin{cond}\label{cond2tau}
  For all $i \ge 1,j \ge 0$ such that $l(i+j)=0$ we have $h(i) \le j$ or $l(j)=0$.   
  \end{cond}
\begin{prop}[Adams--Margolis \cite{am_sub}]
  If the profile function $(h,l)$ satisfies \Cref{cond1tau,cond2tau} then $D(h,l)$ is a Hopf subalgebra of $\cA/\tau$, and moreover every Hopf subalgebra is of this form. 
\end{prop}
\begin{proof}[Sketch of proof.]
  This is not quite proved in \cite{am_sub}, however the proof is analogous. First, it is easy to check, as in \Cref{prop:quotientminimal}, that the coproduct passes to the quotient, so that $D_{\ast,\ast}(h,l) = (\cA/\tau)_{\ast,\ast}/I(h,l)$ is a quotient Hopf algebra, and hence $D(h,l)$ is a Hopf subalgebra. As in \cite{am_sub}, to prove that every Hopf subalgebra is of this form, we define $(P_n)_{\ast,\ast}$ to be the polynomial subalgebra of $\cA_{\ast,\ast}$ generated by $\xi_1,\ldots,\xi_n$, and $(R_n)_{\ast,\ast}$ to be the subalgebra generated by $\xi_1,\xi_2,\ldots$ and $\tau_0,\tau_1,\ldots,\tau_n$. For $n=-1$ we interpret $(R_{-1})_*$ as $P_\infty$ and for $n=\infty$ we interpret $(R_{\infty})_{\ast,\ast}$ as $\cA_{\ast,\ast}$. One then proves the obvious analog of the current proposition for $P_n$ and $R_n$, via an induction on $n$, using the exact arguments given by Adams and Margolis. The only minor thing to keep in mind is that when Adams and Margolis talk about degree, this refers only to the topological degree - the motivic weight does not play a role. 
\end{proof}

Since $(Q_i)^2 = 0$ and $(P_t^s)^2 = 0$ for $s \leq t$ in $\cA/\tau$, then given an $\cA/\tau$-module $M$, we can define Margolis homology groups with respect to these elements:
\[
H(M;Q_i) = \frac{\ker Q_i \colon M \to M}{\im Q_i\colon  M \to M} \quad \text{and} \quad H(M;P_t^s) = \frac{\ker P_t^s \colon M \to M}{\im P_t^s \colon M \to M.}
\]

For the following, we can then essentially quote the proof of the main theorem of \cite{mw_vanishing}. Note that here $(m-1)$-connected refers to the topological degree. 

\begin{thm}[Miller--Wilkerson \cite{mw_vanishing}]\label{thm:amodtmargolis}
Let $A$ be a Hopf subalgebra of $\cA/\tau$, and let $M$ be an $(m-1)$-connected $A$-module. 
\begin{enumerate}[(i)]
\item If $H(M;P_t^s) = 0$ for all $P_t^s \in A$ with $s\leq t$ and $H(M;Q_i)=0$ for all $Q_i \in A$, then $M$ is $A$-free. 
\item If $H(M;P_t^s) =0$ for all $P_t^s \in A$ such that $s\leq t$ and $|P_t^s| < d$ and if $H(M;Q_t) = 0$ for all $Q_t \in A$ such that $|Q_t| < d$, then $\Tor^A_{p,q,w}(\F_2,M) = 0$ whenever $q < dp-c$ with $c$ depending only on $d$ and $m$. Moreover, if $M$ is of finite type, then $\Ext_A^{p,q}(M,\F_2) = 0$ for $q<dp-c$. 
\end{enumerate}
\end{thm}
\begin{proof}[Sketch of proof.]
  As noted, we can simply quote Miller--Wilkerson, but for the benefit of the reader we outline the argument. Following their lead, we only prove (ii) and leave (i) for the reader. The first step is to reduce to $A$ being finite-dimensional using \cite[Prop.~3.2]{mw_vanishing}. The proof is then by induction on the $\F_2$-dimension of $A$, using the observation that any finite-dimensional Hopf subalgebra of $\cA/\tau$ can be built out of extensions of the form 
  \[
B \to A \to E[x],
  \]
 where $E[x]$ refers to the Hopf algebra $E[x] = \F_2[x]/(x^2)$. 

Now given the dual of $A$ in the form $D_{\ast,\ast}(h,k)$, we let 
\[
t = \min \{ j \colon h(j) >0 \text{ or } k(j-1)>0 \}. 
\]

In the case where $h(t)=0$, there is an extension of the form 
\[
B \to A \to E[Q_{t-1}]. 
\]
If $h(t)>0$, then we have an extension of the form
\[
B \to A \to E[P_t^s],
\]
and we naturally need to consider the cases where $s\leq t$ and $s > t$ separately. All three cases are handled in the same way as in \cite{mw_vanishing}; perhaps the only thing to check is that Prop.~4.1 of \emph{loc.~cit.} still holds, but it is easy to see that the same proof works in this case.
\end{proof}
\section{Margolis homology, projective modules, and vanishing lines}\label{sec:proof}
In this section, we complete the proof of the main theorem and return to the motivating examples in the introduction. We begin with a definition, followed by a key lemma, essentially appearing in \cite{gir_picard}. 

\begin{defn}
 We say that a bigraded module $M^{t,w}$ is $(m-1)$-connected provided $M^{t,w} = 0$ when $t<m$, and it is connective if it is $(m-1)$-connected for some $m$. $M$ is of finite type if it is connective and each $M^{t,\ast}$ is a finitely generated $\M_2$-module.
\end{defn}
\begin{lem}\label{lem:finite}
  Let $A$ be a connected Hopf algebra over $\M_2$ which is free and of finite type over $\M_2$, and let $M$ be an $A$-module, also free and of finite type over $\M_2$. Then the following conditions are equivalent:
  \begin{enumerate}[(i)]
    \item $M$ is projective as an $A$-module. 
    \item $M/\tau$ is projective as an $A/\tau$-module. 
    \item $M/\tau$ is free as an $A/\tau$-module. 
    \item $M$ is free as an $A$-module. 
  \end{enumerate}
\end{lem}
\begin{proof}
  (i) $\Rightarrow$ (ii): For $M$ an $A$-module, and $N$ an $A/\tau$-module, we have a change-of-rings isomorphism 
  \[
   \Hom_A(M,N) \cong \Hom_{A/\tau}(M/\tau,N)
  \]
  Since $M$ is free over $\M_2$, this derives to an isomorphism
  \[
   \Ext^p_A(M,N) \cong \Ext^p_{A/\tau}(M/\tau,N)
  \]
  Now if $M$ is projective over $A$, this proves that $\Ext^1_{A/\tau}(M/\tau,-)$ vanishes, so $M/\tau$ is projective over $A/\tau$.\\
  (ii) $\Rightarrow$ (iii): For a connected Hopf algebra over a field, projective implies free for connective modules. To see this, for a projective module $N$ form the minimal free resolution $F_i$. This has the property that all the differentials have coefficients in the augmentation ideal of $A/\tau$. So $\Ext_{A/\tau}^p(N,k) \cong \Hom_{A/\tau}(F_p,k)$, and $N$ projective implies $F_1=0$. But then $F_0\to N$ is an isomorphism and $N$ is free. \\
  (iii) $\Rightarrow$ (iv): Choose a basis of $M/\tau$ over $A/\tau$. Lifting that to $M$, we get a map from a corresponding free $A$-module $F\to M$. By assumption, that map induces an isomorphism $F/\tau\to M/\tau$. Since $M$ is $\M_2$-free, $M/\tau \cong \tau^kM/\tau^{k+1}$, and analogously for $F$. By induction via the five-lemma, we see that the induced map $F/\tau^k \to M/\tau^k$ is an isomorphism for any $k$.
  Since $M$ is finite type over $\M_2$, this implies that $F\to M$ is an isomorphism.\\
  (iv) $\Rightarrow$ (i) is clear.
\end{proof}

Finally, we note the following relation between Margolis homology groups in $\cA/\tau$ and in $\cA$.
\begin{lem}\label{lem:marcomparasion}
Let $M$ be an $\cA$-module, which is free and of finite type over $\M_2$, and let $x\in \cA$ be such that $x^2$ acts trivially on $M$. Then $H(M;x) = 0$ if and only if $H(M/\tau;x) = 0$. 
\end{lem}

\begin{proof}
The long exact sequence in Margolis homology associated to the short exact sequence $0 \to M \xr{\cdot \tau} M \to M/\tau \to 0$ shows that if $H(M;x) = 0$, then $H(M/\tau;x) = 0$. Conversely, if $H(M/\tau;x) = 0$, then $H(M;x) \xr{\cdot \tau} H(M;x)$ is an isomorphism. By assumption on $M$, $H(M;x)$ vanishes in low motivic weight (depending on the degree), and since $\tau$ is an isomorphism, it must vanish in all weights, and hence $H(M;x) = 0$.\footnote{This argument is similar to that given in \cite[Lem.~3.4]{gir_picard}.}
\end{proof}

We now give the proof of our main theorem. 
\begin{thm}\label{thm:mainthm}
  Let $(h,k)$ be a free profile function, and $B=B(h,k)$ the corresponding Hopf subalgebra of $\cA$. Let $M$ be a $B$-module that is free and of finite type over $\M_2$. 
  \begin{enumerate}[(i)]
    \item If $H(M/\tau;P_t^s) =0$ for all $P_t^s \in B$ with $s\leq t$ and $H(M/\tau;Q_t) = 0$ for all $Q_t \in B$, then $M$ is $B$-free. 
    \item If $H(M/\tau;P_t^s) =0$ for all $P_t^s \in B$ such that $s\leq t$ and $|P_t^s| < d$ and if $H(M/\tau;Q_t) = 0$ for all $Q_t$ such that $|Q_t| < d$, then $\Ext_B^{p,q}(M,\M_2) = 0$ for all $(p,q)$ such that $q < dp-c$ with $c$ depending only on $d$ and the connectivity of $M$. 
  \end{enumerate}
\end{thm}
\begin{proof}
 
\begin{enumerate}[(i)]
  \item By \Cref{lem:finite} it suffices to show that $M/\tau$ is free as a $B/\tau$-module. It is easy to see that $P^s_t \in B$ if and only if $P^s_t \in B/\tau$, and similar for $Q_t$. The result then follows from \Cref{thm:amodtmargolis}. 

  \item By \Cref{thm:amodtmargolis} $\Ext^{p,q}_{B/\tau}(M/\tau,\F_2)$ has a vanishing line of the claimed form. Associated to the filtration of $B$ by powers of $\tau$ there is a trigraded-graded algebraic Miller--Novikov type Bockstein spectral sequence
  \[
E_1 = \Ext^{\ast,\ast}_{B/\tau}(M/\tau,\F_2)[\tau] \implies \Ext^{\ast,\ast}_{B}(M,\M_2). 
  \]
This is convergent by the assumptions on $M$ (see \cite[Sec.~8]{miller_relations}). Since $\tau$ has internal degree 0, the vanishing line of $\Ext^{p,q}_{B/\tau}(\M/\tau,\F_2)$ extends to the $E_1$-page of the spectral sequence. Differentials in a spectral sequence cannot increase the vanishing line, and the result follows. \qedhere
\end{enumerate}
\end{proof}
\begin{rem}\label{rem:otherfields}
   One can ask for generalizations of this result to other base fields $k$. Let $p$ be a prime, and $\ell \ne p$ the characteristic of $k$, so that we work with the algebra of bistable operations in the motivic cohomology of smooth $k$-schemes with $\Z/p$-coefficients. We start with the case $p=2$. In this case, it is known from Voevodsky's solution of the Milnor conjecture \cite{voe_2} that $H^{*,*}(\Spec k,\Z/2) \cong k_*^M(k)[\tau]$, where $k_*^M(k) = K^M_*(k)/2$ is the mod 2 Milnor $K$-theory of $k$.  If $k$ is algebraically closed, then $K_i^M(k)$ is known to be divisible for $i \ge 1$ \cite[III.~Ex.~7.2(b)]{weibelkbook} and $K_0^M(k) \cong \Z$, and so $k_*^M(k) \cong \Z/2$ concentrated in degree 0. It follows that the motivic cohomology of a point $H^{*,*}(\Spec k,\Z/2) \cong \F_2[\tau] \cong \M_2$. Moreover, the Steenrod algebra has the same form as presented in \Cref{sec:background} (in positive characteristic see \cite[Thm~5.6]{hko}) and hence our main theorem applies also to these fields. 

On the other hand, we do not yet see a way to extend the results to more general fields, even to $\mathbb{R}$. In this case $\cA$ is no longer a Hopf algebra, but rather a Hopf algebroid, and the dual Steenrod algebra becomes more complicated; there exist two distinguished elements $\tau$ and $\rho$ in the motivic cohomology of a point, and the dual Steenrod algebra takes the form 
  \[
\cA_{\ast,\ast} \cong H_{\ast,\ast}(\Spec k,\Z/2)[\tau_i,\xi_{i+1} \mid i \ge 0]/(\tau_i^2 + \tau \xi_{i+1} + \rho\tau_{i+1} + \rho \tau_0 \xi_{i+1}).
\]
 Additionally the element $\tau$ is no longer invariant, since its right unit is given by $\eta_R(\tau) = \tau + \tau_0 \rho$. One can quotient by the ideal $(\rho,\tau)$, but in this case we do not know how to prove an analog of \Cref{lem:finite}. Moreover, $\rho$ has bidegree $(1,1)$ ($(-1,-1)$ in homological grading conventions) and so a Bockstein spectral sequence from $\cA/\rho$ to $\cA$ can affect the vanishing line, and one would need to understand the differentials in the Bockstein spectral sequence in more detail.

If $p$ is odd, then we can make a similar conclusion as to above. In general, one has that the motivic Steenrod algebra is isomorphic to the odd-primary Steenrod algebra base-changed to the ring $H^{*,*}(\Spec(k),\Z/p)$, see \cite[Thm.~12.6]{voe_1} for characteristic 0, and \cite[Thm~5.6]{hko} for fields of positive characteristic.  If $k$ is algebraically closed, then the motivic cohomology of a point is $\F_p[\tau]$, and the motivic Steenrod algebra is simply the classic odd-primary Steenrod algebra base changed to $\F_p[\tau]$. In this case it is easy to prove a version of our main theorem. When $k$ is not algebraically closed the situation may still be tractable at odd primes; for example, in \cite[Prop.~1.1(2)]{1606.06085} Stahn shows that the motivic cohomology of $\Spec(\mathbb{R})$ is isomorphic to $\Z/p[\theta]$, with $|\theta| = (0,2)$, and one can also prove a version of our main theorem in this case. 
\end{rem}

We now return to the examples given in the introduction. 
\begin{exmp}\cite[Prop.~4.7]{gir_picard}
  \normalfont Let $\cA(1)$ be the Hopf subalgebra generated by $\Sq^1$ and $\Sq^2$, with dual 
\[
\cA(1)_{\ast,\ast} \cong \M_2[\xi_1,\tau_0,\tau_1]/(\xi_1^2, \tau_1^2,\tau_0^2 + \tau \xi_1). 
\]
We have $Q_0,Q_1$ and $P_1^1$ in $\cA(1)$, and so we see that a finite type $\cA(1)$-module $M$ is free only if:
\begin{enumerate}[(i)]
  \itemsep-0.3em
  \item $M$ is a free $\M_2$-module;
  \item $H(M/\tau;Q_0) = 0$;
  \item $H(M/\tau;Q_1) = 0$;
  \item $H(M/\tau;P_1^1) = 0$,
\end{enumerate}
which is in fact a slight strengthening of the result in \emph{loc.~cit.} (where $M$ is assumed to be finitely generated over $\M_2$). The converse also follows by an easy calculation. 

Our results also cover \cite[Prop.~3.2]{gi_vanishing}, namely that a bounded below $\cA(0)$-module of finite type is free as an $\cA(0)$-module if and only if it is free as an $\M_2$-module and $H(M;Q_0) = 0$; here we need to use \Cref{lem:marcomparasion} to lift from $H(M/\tau;Q_0) = 0$ to $H(M;Q_0)$ = 0. 
\end{exmp}
\begin{exmp}\cite{gi_vanishing}\label{ex:gi_vanishing}
\normalfont
Although it is not proved explicitly, the method of Guillou and Isaksen can easily be used to show that if $M$ is a bounded-below $\cA$-module of finite type that is free as a $\cA(0)$-module, then $\Ext_{\cA}^{p,q,w}(M,\M_2)$ has a vanishing line of slope 1.

To recover this from our work, observe that since $H(M/\tau,Q_0) = 0$, \Cref{thm:mainthm} applies with $d=2$, so that $\Ext^{p,q,w}_{\cA}(M,\M_2) = 0$ whenever $q < 2p -c$ (i.e. $M$ has a vanishing line of slope 1). 

Guillou and Isaksen prove in \cite[Prop.~7.2]{gi_vanishing} that for such an $M$ the map 
\begin{equation}\label{eq:h1}
h_1 \colon \Ext^{p,q,w}_{\cA}(M,\M_2) \to \Ext_{\cA}^{p+1,q+2,w+1}(M,\M_2)
\end{equation}
is an isomorphism if $q \le 3p-4$ and a surjection if $q \le 3p-1$. 
\begin{rem}
  Note that we use a different grading convention to \cite{gi_vanishing} - a class we write in $\Ext^{p,q,w}$ corresponds to a class in $\Ext^{q-p,p,w}$ under their conventions.
\end{rem}
To recover this from our results (at least the slope - we are not precise about the exact value of the constant), note that corresponding to $h_1 \in \Ext^{1,2,1}_{\cA}(M,M)$ there is an exact sequence of modules
\[
0 \to M \to M/\!/h_1 \to \Sigma^{-2,-1}M \to 0.
\]
It is easy to see that $(M/\!/h_1)/\tau$ has vanishing Margolis homologies with respect to $Q_0$ and $P_1^1$, so that $\Ext^{p,q,w}_{\cA}(M/\!/h_1,\M_2) = 0$ whenever $q < 3p -c'$ (i.e. $M/\!/h_1$ has a vanishing line of slope 1/2).  The long exact sequence
\[
\cdots \xr{h_1} \Ext_{\cA}^{p,q+2,w+1}(M,\M_2) \to \Ext_{\cA}^{p,q,w}(M/\!/h_1,\M_2)\to \Ext_{\cA}^{p,q,w}(M,\M_2) \xr{h_1}\cdots 
\]
implies that \eqref{eq:h1} is an injection when $q < 3p-c'$ and a surjection when $q < 3p+3-c'$. Hence, we see that $h_1$-multiplication is an isomorphism above a line of slope 1/2 as claimed. 
\end{exmp}
\bibliography{bib_motivic} 
\bibliographystyle{amsalpha}
\end{document}